\documentclass[12pt]{amsart}
\usepackage{bbm,amssymb,amsmath,  accents, stmaryrd}

\usepackage{amsmath}
\usepackage{amscd}
\usepackage{amssymb}

\usepackage{enumerate,array,mathrsfs}
\usepackage{graphicx}
\usepackage{xcolor}

\usepackage{rotating}
\usepackage{bbm}
\usepackage[all,cmtip]{xy}

\usepackage[latin1]{inputenc}
\usepackage{dsfont}

\usepackage{hyperref}

\usepackage{color}

\newtheorem{theorem}{Theorem}[section]
\newtheorem{lemma}[theorem]{Lemma}
\newtheorem{proposition}[theorem]{Proposition}

\theoremstyle{definition}
\newtheorem{definition}[theorem]{Definition}
\newtheorem{remark}[theorem]{Remark}

\numberwithin{equation}{section}

\newcommand{\gras}[1]{{\mathbb #1}}

\newcommand{\R}{\gras{R}}

\newcommand{\OB}{\mathop {OB}}
\newcommand{\id}{\mathop {id}}

\def\elem(#1,#2){  \{ \frac{#1}  {\overline {\ #2\ }} \} }

\title[Contact open books with exotic pages]
    {Contact open books with exotic pages}
\author{Burak Ozbagci}
    \address{Department of Mathematics, Ko\c{c} University, Rumelifeneri Yolu,
         34450, Sariyer, Istanbul, Turkey}
     \email{bozbagci@ku.edu.tr}
\author{Otto van Koert}
   \address{Department of Mathematics and Research Institute of Mathematics, Seoul National University,
Building 27, room 402, San 56-1, Sillim-dong, Gwanak-gu, Seoul,
South Korea, Postal code 151-747}
   \email{okoert@snu.ac.kr}
\begin{document}

\begin{abstract}
   We consider a fixed contact $3$-manifold that admits infinitely many compact Stein fillings which are all homeomorphic but pairwise
   non-diffeomorphic. Each of these fillings gives rise to a closed contact $5$-manifold described as a contact open book whose page is the filling at hand
   and whose monodromy is the identity symplectomorphism. We show that
   the resulting infinitely many contact $5$-manifolds  are all diffeomorphic but pairwise
   non-contactomorphic. Moreover, we explicitly determine these contact $5$-manifolds.

\end{abstract}

\maketitle

%\tableofcontents

\section{Introduction}
\label{intro}

Recent advances in symplectic geometry and topology showed that
while some closed contact $3$-manifolds have only finitely many
Stein fillings, others have infinitely many, up to diffeomorphism
(see \cite{oz} for a recent survey). Among the $4$-manifold
topologists, it is common to call a Stein filling of a contact
$3$-manifold exotic compared to another filling, if these two
fillings are homeomorphic but non-diffeomorphic.

The first examples of a closed contact $3$-manifold admitting
 infinitely many exotic \emph{simply-connected} Stein fillings were
discovered in \cite{aems}. The Stein fillings in that article, and
many others which appeared in the literature  since then, were given
as Lefschetz fibrations over the disk whose boundary is a fixed open
book supporting the contact manifold in question.

Recently,  Akbulut and Yasui \cite{ay} constructed an infinite
family of exotic simply-connected Stein fillings of a fixed contact
$3$-manifold, where the fillings (with $b_2=2$) are described by
explicit handlebody diagrams, rather than Lefschetz fibrations. In
this paper, we consider the infinite set of contact open books each
of which has a fixed exotic Stein filling in their family as its
page and identity as its monodromy, and show that the resulting
closed contact $5$-manifolds are all diffeomorphic but pairwise
non-contactomorphic.

Moreover, we give two alternative arguments to distinguish the
contact $5$-manifolds: one is based on the Barden's classification
\cite{ba} of simply-connected closed $5$-manifolds, and the other is
based on the diagrammatic language developed for contact
$5$-manifolds in \cite{dgvk}. The advantage of the latter approach
is that we can explicitly identify the contact $5$-manifolds.
%there is no restriction on the fundamental group of the $5$-manifold.

\section{Simply-connected exotic Stein fillings}
\label{exot}

We briefly review the infinite family of exotic simply-connected
Stein fillings of a fixed contact $3$-manifold due to Akbulut and
Yasui \cite{ay}. Let $X$ be the $4$-manifold with boundary described
by the handlebody diagram on the left in Figure~\ref{diag}. Note
that there is an embedded $T^2 \times D^2$ in the interior of $X$.
Throughout this section, let $p$ denote a positive integer and let
$X_p$ be the result of $p$-log transform on the $T^2 \times D^2
\subset X$. A handlebody diagram of $X_p$ is presented on the right
in Figure~\ref{diag}.

First of all, one observes that  $X$ and $X_p$ are simply-connected,
for all $p$.  To see this, cancel the upper $1$-/$2$-handle pair and
the lower $1$-/$2$-handle pair  to get a diagram consisting of only
two $2$-handles and no $1$-handles  for both $X$ and $X_p$. This
already implies that $b_2(X)=b_2(X_p)=2$, and allows one to easily
compute the intersection forms of $X$ and $X_p$, which turns out to
be unimodular and indefinite, for all $p$. Such forms are
classified, up to isomorphism, by their rank, signature and parity.
The signature of all the forms are zero, the form of $X$ is even,
and the form of $X_p$ is even if and only if $p$ is odd. Thus, by
Boyer's generalization \cite{bo} of Freedman's celebrated theorem
\cite{f}, one concludes that $X_p$ is homeomorphic to $X_{p'}$ if
and only if $p$ and $p'$ have the same parity and $X_p$ is
homeomorphic to $X$ if and only if $p$ is odd. Moreover, for all
$p$,  $\partial X=
\partial X_p$ is a homology $3$-sphere.

\begin{figure}[htp]
\def\svgwidth{0.75\textwidth}%
\begingroup\endlinechar=-1
\resizebox{0.75\textwidth}{!}{%
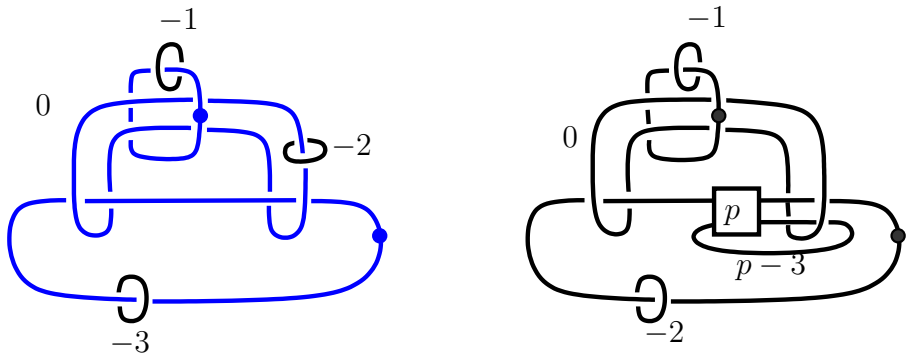%
}\endgroup
\caption{On the left: the handlebody diagram of $X$, where $T^2 \times D^2 \subset X$ is
        indicated with blue color. On the right: the handlebody
        diagram of $X_p$.}
\label{diag}
\end{figure}

Next, one shows that $X$ and $X_p$ admit Stein structures, for all
$p$, by turning the smooth handlebody diagrams in Figure~\ref{diag}
into Legendrian handlebody diagrams (see Figure~\ref{fig:stein}),
after cancelling the upper $1$-/$2$-handle pairs for convenience.

\begin{figure}[htp]
\def\svgwidth{0.75\textwidth}%
\begingroup\endlinechar=-1
\resizebox{0.75\textwidth}{!}{%
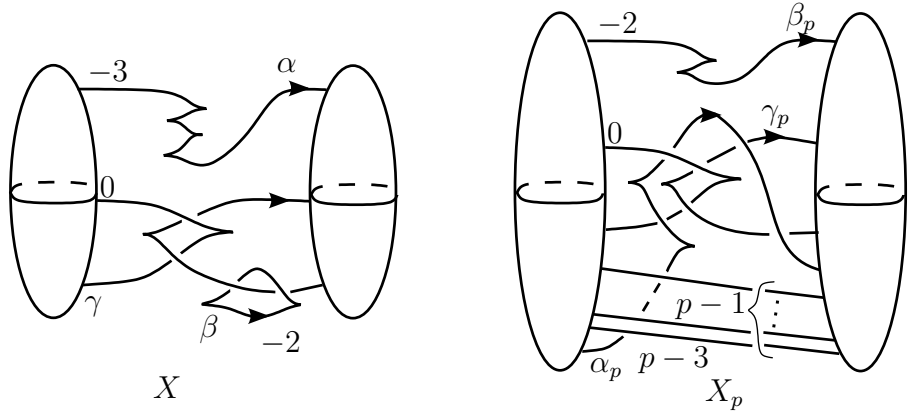%
}\endgroup \caption{Stein handlebody diagrams for $X$ and $X_p$ }
\label{fig:stein}
\end{figure}

In particular, $\partial X=\partial X_p$ admits a Stein fillable
contact structure. Finally, one uses the adjunction inequality
coupled with the genus function to distinguish the smooth structures
on $X_p$'s:

\begin{proposition}[Akbulut and Yasui] \label{exotpair} (i) There exists a contact structure $\eta$ on $\partial X$ such that
for infinitely many values of $q \geq 1$, $X_{2q-1}$ is a Stein
filling of $(\partial X, \eta)$ with the property that all of these
fillings are homeomorphic  but pairwise non-diffeomorphic.

(ii) There exists a contact structure $\eta'$ on $\partial X$ such
that for infinitely many values of $q \geq 1$, $X_{2q}$ is a Stein
filling of $(\partial X, \eta')$ with the property that all of these
fillings are homeomorphic  but pairwise non-diffeomorphic.
\end{proposition}

\section{Contact open books with exotic pages}
\label{sec:Contopen} Let $M_p:=OB(X_p, id)$ (resp. $M:=OB(X, id)$)
denote the closed contact $5$-manifold viewed as the contact open
book with page $X_p$ (resp. $X$) and monodromy the identity map. Let
$\xi_p$ (resp. $\xi$) denote the contact structure supported by this
open book in the sense of Giroux \cite{g}.

We start with some basic observations about the contact
$5$-manifolds $(M, \xi)$ and $(M_p, \xi_p)$. First of all, $(M,
\xi)$ and $(M_p, \xi_p)$ are subcritically Stein fillable, and in
fact $W:=X\times D^2$ and $W_p:=X_p\times D^2$ are their subcritical
Stein fillings, respectively \cite[Prop. 3.1]{dgvk}.

Since $X$ and $X_p$ are simply-connected for all $p$, so are $W$ and
$W_p$. Moreover, $M$ and $M_p$ are simply-connected for all $p$, as
a result of the following simple but useful lemma. Note that a
compact Stein filling (a.k.a. a Stein domain) is a Weinstein domain
(cf. \cite{ce}).

\begin{lemma}
If $V^{2n}$ is a Weinstein domain, then the inclusion map
$i:\partial V \to V$ induces an isomorphism on $\pi_1$, provided
that $n\geq 3$.

\end{lemma}

\begin{proof}
Fix a base point $b$ in the boundary $\partial V$. Suppose that
$\gamma$ is a loop in $W$. We claim that $\gamma$ is homotopic to a
loop in $\partial V$, or in other words, that $i$ induces a
surjective map on $\pi_1$. First we perturb $\gamma$, a
$1$-dimensional CW-complex, to make it disjoint from the isotropic
skeleton of $V$, which has dimension at most $n$. Then apply the
Liouville flow, cut off near the boundary, to push $\gamma$ into a
collar neighborhood of the boundary.

For injectivity on $\pi_1$, suppose that $\gamma_1$ and $\gamma_2$
are loops in $\partial V$ that are homotopic in $V$. Such a homotopy
is a $2$-dimensional CW-complex, so we can make it disjoint from the
isotropic skeleton if $n\geq 3$. Then apply the Liouville flow to
push the entire homotopy into a collar neighborhood of the boundary.
\end{proof}

\begin{remark} When $n=2$, the map $i_*: \pi_1(\partial V) \to \pi_1 (V)$   is surjective \cite[Prop. 1.10]{st}, but not necessarily injective
as, for instance, $T^2 \times D^2$ is a Stein (hence Weinstein)
filling of the standard contact $T^3$.
\end{remark}

\begin{lemma}
\label{lemma:bar}  Let $S^2\tilde \times S^3$ denote the non-trivial
$S^3$-bundle over $S^2$. Then $M_p$ is diffeomorphic to either
$S^2\times S^3 \# S^2\times S^3$ or $ S^2\times S^3 \# S^2\tilde
\times S^3$.
\end{lemma}
\begin{proof} For the subcritical Stein filling $W_p$, the homology sequence$$
\underset{\cong H^3(W_p)=0}{H_3(W_p,\partial W_p)} \longrightarrow
H_2(\partial W_p) \stackrel{i_*}{\longrightarrow} H_2( W_p)
\longrightarrow \underset{\cong H^4(W_p)=0}{H_2(W_p,\partial W_p)},
$$ of the pair $(W_p,
\partial W_p)$ implies that $i_*$ is an isomorphism, since the homology of
subcritical Weinstein manifolds vanishes in degree at least half the
dimension. As a consequence we have
$$H_2 (M_p) \cong H_2(W_p)\cong H_2(X_p)\cong \mathbb{Z} \oplus \mathbb{Z}.$$

The statement in the lemma follows from Barden's classification
\cite{ba} of diffeomorphism classes of simply-connected closed
$5$-manifolds. Note that $S^2\tilde \times S^3 \# S^2\tilde \times
S^3$ is diffeomorphic to $S^2\times S^3 \# S^2 \tilde \times S^3$.
\end{proof}

Next, in order to identify the contact $5$-manifold $(M_p, \xi_p)$,
for all $p$, we would like to determine the first Chern class
$c_1(\xi_p)$ using the following  general results. Let $c$
denote the total Chern class.

\begin{lemma}
\label{lemma:c1_strong_filling} Suppose that $(V,\omega)$ is a
strong symplectic filling of some closed contact manifold $(Y,\xi)$.
If $i: Y\to V$ denotes the inclusion map, then $i^*c(TV)=c(\xi)$.
\end{lemma}
\begin{proof}
There is a Liouville vector field $Z$ defined near $\partial V$ with
the Liouville form $\lambda:=i_Z \omega$. Its restriction
$\alpha:=\lambda|_{\partial V}$ defines a contact form. We get a map
$((-\epsilon,0]\times Y,d(e^t \alpha)\, )\to (V,\omega)$ by sending
$(t,p)$ to $Fl^Z_t(i(p))$, where $Fl^Z$ denotes the flow induced by
$Z$. This map preserves the symplectic structure. We conclude that
$TV|_{\partial V=Y}$ symplectically splits as
$$
(span_\R(Z,R) \oplus \xi,\omega_0\oplus \omega|_{\xi}).
$$
Here $R$ is Reeb vector field on the contact manifold $(Y,\alpha)$,
and $\omega_0$ defines the standard symplectic form. Note that $Z$,
$R$ forms a symplectic frame as $\omega(Z,R)=1$. The rank $2$
symplectic vector bundle $\epsilon=span_\R(Z,R)$ is clearly trivial,
so we see that
$$
c(i^*TV)=c(\epsilon\oplus \xi)=c(\epsilon)c(\xi)=c(\xi).
$$ \end{proof}

%The following simple claim is useful too, as it gives $c_1$.
\begin{lemma} \label{chern}
Suppose $V^{2n}$ is a Weinstein domain with $n\geq 3$. Assume in
addition that $W$ is subcritical for $n=3$.  Then the inclusion map
$i:\partial V\to V$ induces an isomorphism on $H^2$. In particular,
if $(Y,\xi)$ is the contact boundary of $V$, then $c_1(\xi)$
determines and is determined by $c_1(TV)$.
\end{lemma}
\begin{proof}
We just consider the long exact sequence of cohomology groups
$$
\underset{\cong H^{2n-2}(V)=0}{H^2(V,\partial V)} \longrightarrow
H^2(V) \stackrel{i^*}{\longrightarrow} H^2(\partial V)
\longrightarrow \underset{\cong H^{2n-3}(V)=0}{H^3(V,\partial V)}
$$
of the pair $(V, \partial V)$ to conclude that $i^*$ is an
isomorphism, assuming in addition that $V$ is subcritical for the
case $n=3$. Lemma~\ref{lemma:c1_strong_filling} then shows the last
claim.
\end{proof}

Finally, we are ready to prove the main result of the paper.

\begin{theorem} \label{mainthm}
If $p$ is odd, then $M_p$ (and hence $M$) is diffeomorphic to
$S^2\times S^3 \# S^2\times S^3$. If $p$ is even, then $M_p$ is
diffeomorphic to $S^2\times S^3 \# S^2 \tilde \times S^3$.
Furthermore, $(M_p, \xi_p)$ is contactomorphic to $(M_{p'},
\xi_{p'})$ if and only if $p=p'$.
\end{theorem}

\begin{proof}

We know that the first Chern class of any Stein surface  can be
calculated using an explicit Legendrian handlebody diagram
representing the surface \cite[Prop. 2.3]{go}. Let $\alpha_p,
\beta_p$, and $\gamma_p$ be the Legendrian curves in the handlebody
diagram for $X_p$ depicted in Figure~\ref{fig:stein} and let
$T_p=[\gamma_p]$ and $R_p= [\alpha_p] - p [\beta_p] \in H_2(X_p) $.
As it was observed in \cite{ay}, $H_2(X_p)$ has a basis consisting
of $T_p$ and $S_p$, where
\begin{displaymath}
   S_p = \left\{
     \begin{array}{lr}
       R_{q-1} + \big((2q-1)^2-q+1\big)T_{2q-1}  & : p=2q-1 \\
       R_{2q} + \big((2q)^2-q+1\big)T_{2q}  & : p=2q
     \end{array}
   \right.
\end{displaymath}
It is easy to see that $c_1(TX_p)$ evaluates on these homology
classes as:
$$
\langle c_1(TX_p),S_p\rangle=-1-p,\quad \langle
c_1(TX_p),T_p\rangle=0.
$$

The first Chern class $c_1 (\xi_p) $ can be viewed as a linear map
from $H_2(M_p)\cong \mathbb{Z} \oplus \mathbb{Z}$ to $\mathbb{Z}$.
The cycles $S_p \times \{1\}$, and  $T_p \times \{1\}$  in $M_p=
\partial (X_p\times D^2)$, where we think of $1\in \partial D^2$,
generate $H_2(M_p)$.  Using  Lemma~\ref{chern}, we conclude that
$c_1 (\xi_p)$ evaluates as $(-1-p,0)$ with respect to the chosen
basis of $H_2(M_p)$, which is sufficient to mutually distinguish
$\xi_p$'s.

The calculation above shows that $M_p$ is spin for odd $p$ and
non-spin otherwise. We conclude, by Lemma~\ref{lemma:bar}, that
$M_p$  is diffeomorphic to $S^2\times S^3 \# S^2\times S^3$ for odd
$p$, and to $S^2 \times S^3 \# S^2 \tilde \times S^3$, otherwise.
\end{proof}

We can say more explicitly what contact manifold we get with the following lemma.
\begin{lemma}\cite[Prop. 4.5]{dgvk}
\label{lemma:smooth_type} Suppose that $V$ is a Stein surface
obtained by attaching a single $2$-handle to the standard Stein
domain $D^4$ along a Legendrian knot $\delta$ in the standard tight
$S^3$. Then $\OB(V,\id)$ is diffeomorphic to
\begin{itemize}
\item $S^2\times S^3$ if $rot(\delta)$ is even
\item $S^2\tilde \times S^3$ if $rot(\delta)$ is odd.
\end{itemize}
Moreover, if $V'$ is another Stein surface obtained as above using a
Legendrian knot $\delta'$, then $\OB(V,\id)$ and $\OB(V',\id)$ are
contactomorphic if and only if $|rot(\delta)|=|rot(\delta')|$.
\end{lemma}
For each integer $k$, choose a Legendrian unknot $\delta_k$ in the
tight $S^3$ with rotation number equal to $k$. We view this $S^3$ as
the boundary of the Stein domain $D^4$. Define $V_k$ as the
handlebody obtained by attaching a Weinstein $2$-handle to $D^4$
along $\delta_k$.

Then $\partial( V_k \times D^2)$ is an $S^3$-bundle over $S^2$, and
it is diffeomorphic to $S^2\times S^3$ if $k$ is even and to
$S^2\tilde \times S^3$ if $k$ is odd. Up to contactomorphism, the
contact structure only depends on $|k|$ by
Lemma~\ref{lemma:smooth_type}. We denote the resulting contact
structure by $\zeta_k$.

\begin{proposition} The contact $5$-manifold $(M_p, \xi_p)$ is
contactomorphic to
\begin{itemize}
\item $(S^2\times S^3,\zeta_0)\# (S^2\times S^3,\zeta_{p+1})$ if $p$ is odd
\item $(S^2\times S^3,\zeta_0)\# (S^2\tilde \times S^3,\zeta_{p+1})$ if $p$ is even.
\end{itemize}
\end{proposition}
\begin{proof}
A complete argument is given in the proof of
Proposition~\ref{prop:direct_moves}. \end{proof}

\section{An alternative argument}
As in~\cite{dgvk} we encode a closed contact $5$-manifold (described
as a contact open book) by a handlebody diagram for its page---which
can be assumed to be a compact Stein domain. We will only consider
situations where the symplectic monodromy is a product of Dehn
twists along Lagrangian spheres.

\begin{definition}
Two Stein surfaces $X$ and $X'$ are called {\bf contact stably
equivalent} if there are handlebody diagrams for $X$ and $X'$, and a
third handlebody diagram for some Stein surface $X''$ with the
property that the handlebody diagrams for $X$ and $X'$ can be
transformed into the one for $X''$ by a finite sequence of the
following moves:
\begin{itemize}
\item usual handlebody moves for Stein surfaces (cf. \cite{go})
\item stabilizing the attaching circles for the $2$-handles (move I)
\item changing a crossing (move II)
\end{itemize}
\end{definition}
See Figure~\ref{fig:moves} for a description of moves I and II.
\begin{figure}[htp]
\def\svgwidth{1.0\textwidth}%
\begingroup\endlinechar=-1
\resizebox{1.0\textwidth}{!}{%
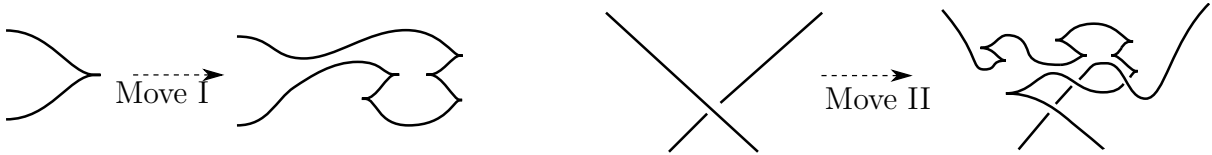%
}\endgroup \caption{Moves I and II give contactomorphic $5$-manifolds} \label{fig:moves}
\end{figure}

With the arguments from \cite[Section 4.1]{dgvk} we obtain the following proposition.
\begin{proposition}
Suppose that $X$ and $X'$ are contact stably equivalent. Then $X
\times D^2$ and $X' \times D^2$ are symplectically deformation
equivalent with contactomorphic boundaries.

\end{proposition}
We briefly summarize \cite[Section 4.1]{dgvk} and explain why moves I and II give rise to contactomorphic manifolds.
Given a contact open book $\OB(W,\id)$ we first positively stabilize the open book.
This is done by first taking any properly embedded Lagrangian disk $D$. We attach a
Weinstein $2$-handle along $\partial D$ to obtain a new page $\tilde W$, which contains a Lagrangian
sphere $L$ formed by gluing the core of the $2$-handle to the Lagrangian disk $D$.
We can then perform a right-handed Dehn twist along $L$.
We denote this Dehn twist by $\tau_L$.
According to Giroux, the contact open book $\OB(\tilde W,\tau_L)$ is contactomorphic to $\OB(W,\id)$.
We use the extra $2$-handle in the page to perform a handle slide.
We then destabilize the open book by simply not performing the Dehn twist and removing the extra $2$-handle.
Since the monodromy on $W$ was assumed to be the identity this last step can be done. See \cite[Section 4.1]{dgvk} for a detailed description of moves I and II and a proof that these moves do not change the symplectic deformation type of the filling.

The upshot is that move I changes the Thurston-Bennequin invariant of an attaching circle of a $2$-handle while preserving its rotation number.
Move II changes an overcrossing into an undercrossing.
This move can be used to change the knot type of the attaching circles.

\begin{proposition}
\label{prop:direct_moves} The contact $5$-manifolds $(M_p, \xi_p)$
and $(M_{p'}, \xi_{p'})$ are diffeomorphic if and only if $p \equiv
p'\;(mod \;2)$, and contactomorphic if and only if $p=p'$.
\end{proposition}
\begin{proof}
%We prove the proposition by constructing the required diffeomorphism and contactomorphism under the given assumptions.

In the following, we refer to \cite[Figures~3 and~9]{go} for
Legendrian Reidemeister moves. We consider the handlebody diagram of
$X_p$ in Figure~\ref{fig:stein} and apply a Legendrian isotopy to
obtain step 1 in Figure~\ref{fig:contactomorphism}. Then we apply
move II simultaneously to the two overcrossings in the shaded region
in step 1. Next we apply a Legendrian Reidemeister move 2 to the
indicated cusp in step 2 and another Legendrian Reidemeister move 2
to the indicated cusp in step 3. We do this twice more, and then
apply Legendrian Reidemeister move 4 to move the Legendrian knot
$\gamma_p$ off the $1$-handle. Finally, by several applications of
Reidemeister move 2 we obtain step 5 of
Figure~\ref{fig:contactomorphism}.

\begin{figure}[htp]
\def\svgwidth{0.7\textwidth}%
\begingroup\endlinechar=-1
\resizebox{0.7\textwidth}{!}{%
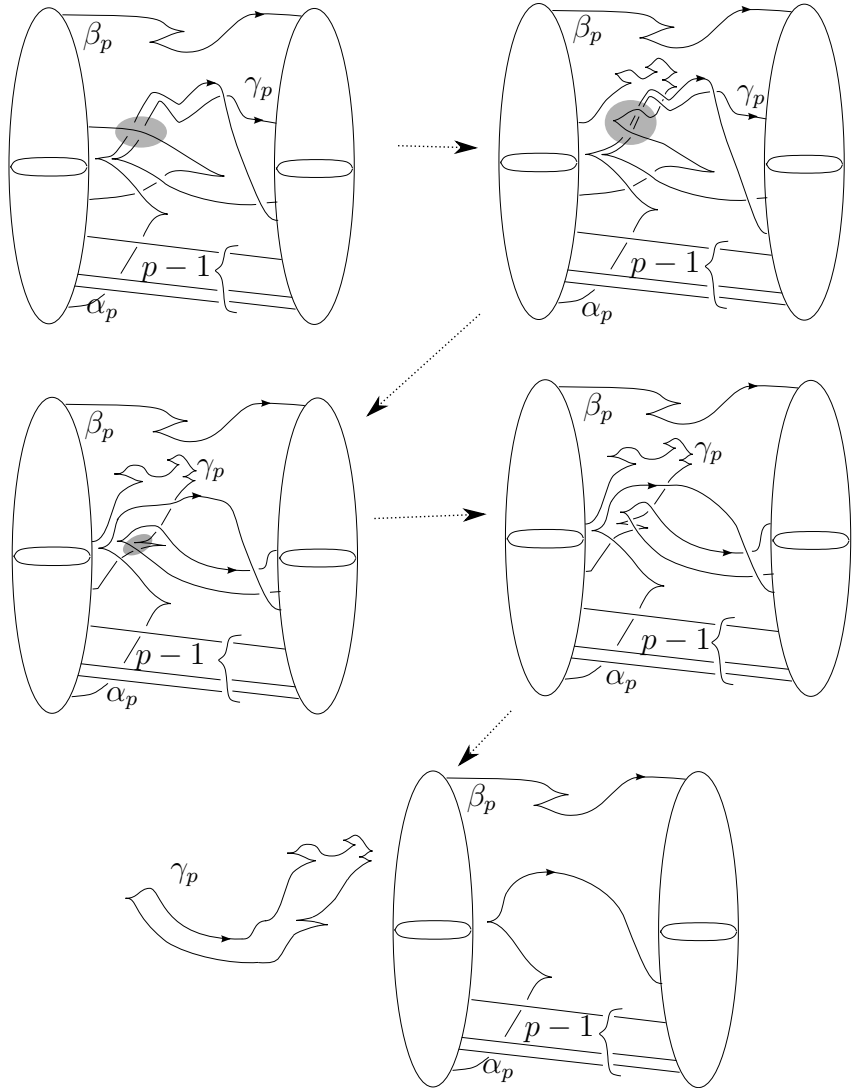%
}\endgroup \caption{Constructing a contactomorphism between $5$-manifolds} \label{fig:contactomorphism}
\end{figure}

We now perform some moves so that we will be able to cancel the
$1$-handle. Note that the dotted curve near $\beta_p$ depicted in
the initial diagram in Figure~\ref{fig:strans} is not part of the
handlebody and it will be used for a handle-slide. First we perform
an isotopy of $\alpha_p$, and then in step 2, we slide the
$\alpha_p$ handle over the $\beta_p$ handle using handle
subtraction. We keep calling the resulting curve the
$\alpha_p$-curve. Note that its rotation number has decreased by $1$
after the handle subtraction.

Now we use move II to unlink the curves $\alpha_p$ and $\beta_p$. We
get additional cusps which we remove with the inverse of move I.
This move keeps the rotation number unchanged. Now we perform
Legendrian Reidemeister move 4 to move one strand of the curve
$\alpha_p$ off the $1$-handle.

We repeat this procedure until there are no strands of $\alpha_p$
left which is going over the $1$-handle. To do this, we apply
Reidemeister move 2 to the indicated cusp and apply an isotopy to
move this cusp in the position of step 1. With move II we can make
the curve $\alpha_p$ into an unknot without changing the rotation
number, so we end up with an unknotted curve whose rotation number
is equal to $-1-p$. With Legendrian Reidemeister move 4 we make the
curve $\alpha_p$ disjoint from the $1$-handle, and then cancel the
$1$-handle with the $2$-handle attached along $\beta_p$. We end up
with a handlebody diagram containing only $2$-handles attached along
the curve  $\gamma_p$ with rotation number $0$ and to the curve
$\alpha_p$ which has rotation number $-1-p$.

Now apply Lemma~\ref{lemma:smooth_type} to see that $M_p$ is contactomorphic to
\begin{itemize}
\item $(S^2\times S^3,\zeta_0)\# (S^2\times S^3,\zeta_{p+1})$ if $p$ is odd
\item $(S^2\times S^3,\zeta_0)\# (S^2\tilde \times S^3,\zeta_{p+1})$ if $p$ is even.
\end{itemize}
By forgetting the contact structure we get the diffeomorphism claim.

\begin{figure}[htp]
\def\svgwidth{1.0\textwidth}%
\begingroup\endlinechar=-1
\resizebox{1.0\textwidth}{!}{%
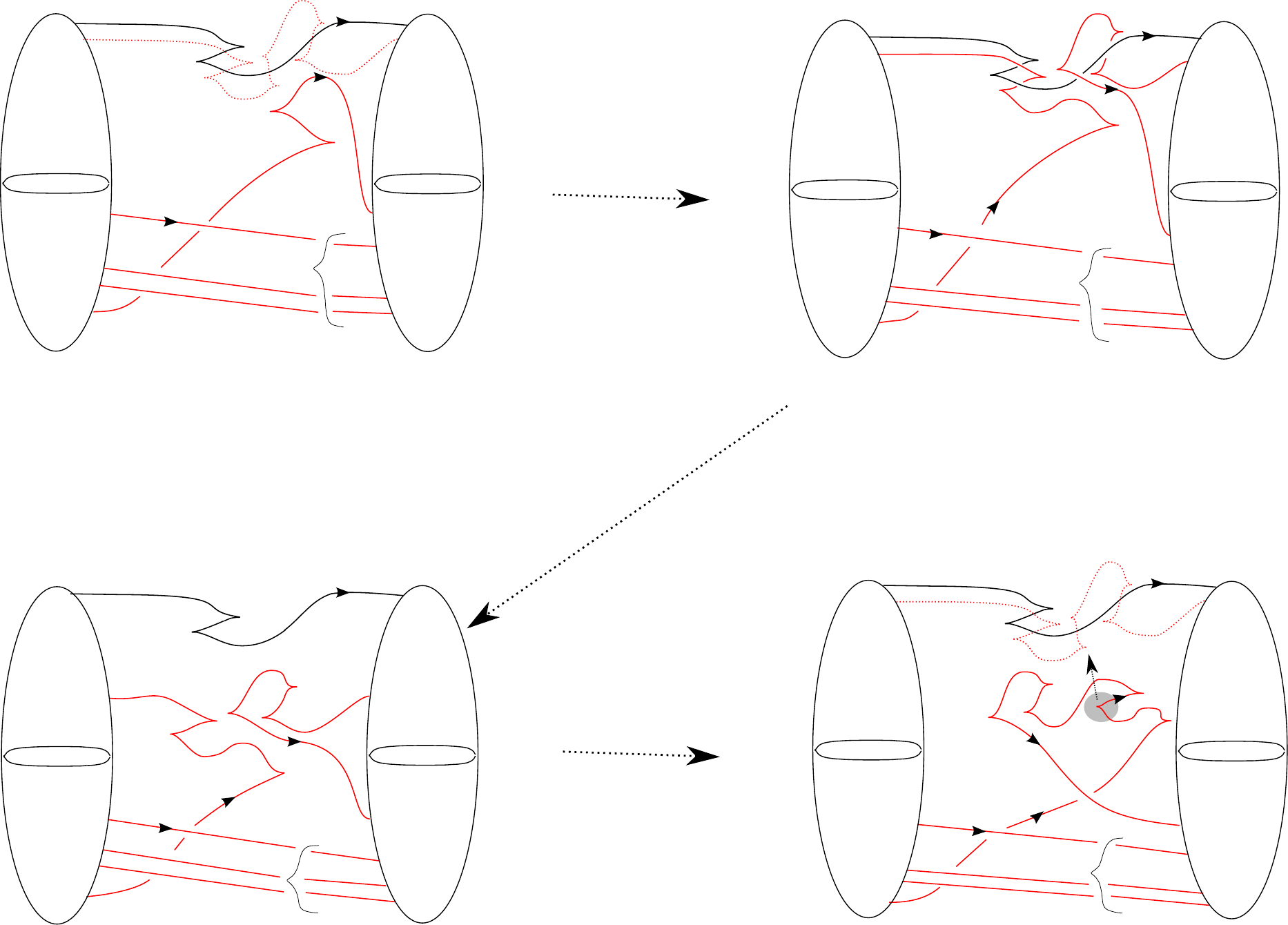%
}\endgroup \caption{Sliding strands off the $1$-handle} \label{fig:strans}
\end{figure}
\end{proof}


\begin{thebibliography}{99999}

\bibitem{ay}
S. Akbulut and K. Yasui, \emph{Infinitely many small exotic Stein
fillings.} J. Symplectic Geom., to appear.

\bibitem{aems}
A. Akhmedov, J.B. Etnyre, T.E. Mark and I. Smith, \emph{A note on
Stein fillings of contact manifolds.} Math. Res. Lett. 15 (2008),
no. 6, 1127-1132.

\bibitem{ba} D. Barden, \emph{Simply connected five-manifolds.} Ann. of Math. (2)  82
(1965), 365-385.


\bibitem{bo} S. Boyer, \emph{Simply-connected 4-manifolds with a given
boundary.} Trans. Amer. Math. Soc.  298  (1986),  no. 1, 331-357.

\bibitem{ce} K. Cieliebak and Y. Eliashberg, \emph{From Stein to Weinstein and back.
Symplectic geometry of affine complex manifolds.} American
Mathematical Society Colloquium Publications, 59. American
Mathematical Society, Providence, RI, 2012.

\bibitem{dgvk} F. Ding, H. Geiges and  O. van Koert, \emph{Diagrams for contact
5-manifolds.} J. Lond. Math. Soc. (2)  86  (2012),  no. 3, 657-682.

\bibitem{f} M. H. Freedman, \emph{The topology of four-dimensional
manifolds.} J. Differential Geom.  17  (1982), no. 3, 357-453.

\bibitem{g} E. Giroux, \emph{G\'{e}om\'{e}trie de contact: de la dimension trois vers
les dimensions sup\'{e}rieures.} (French) [Contact geometry: from
dimension three to higher dimensions]  Proceedings of the
International Congress of Mathematicians, Vol. II (Beijing, 2002),
405-414, Higher Ed. Press, Beijing, 2002.

\bibitem{go} R. E. Gompf, \emph{Handlebody construction of Stein surfaces.} Ann. of
Math. (2)  148  (1998),  no. 2, 619-693.


\bibitem{oz}
B. Ozbagci, \emph{On the topology of fillings of contact
3-manifolds.} Geometry \& Topology Monographs Vol. 19., to appear.


\bibitem{st} A. I. Stipsicz, \emph{Gauge theory and Stein fillings of certain
3-manifolds.} Turkish J. Math.  26  (2002),  no. 1, 115-130.


%\bibitem{v}  O. van Koert, \emph{Open books on contact five-manifolds.} Ann. Inst. Fourier (Grenoble)  58  (2008),  no. 1, 139-157.

\end{thebibliography}
\end{document}